\documentclass[12pt]{article}

\usepackage[latin1]{inputenc}
\usepackage[english]{babel}
\usepackage[T1]{fontenc}
\usepackage{lmodern}
\usepackage{graphicx, xcolor, mathrsfs, ulem}
\usepackage{amsfonts}
\usepackage{latexsym}
\usepackage{stmaryrd}
\usepackage[nice]{nicefrac}
\usepackage{amsmath,amsthm,amssymb}
\usepackage{enumerate}
\usepackage[colorlinks=true,citecolor=black,linkcolor=black,urlcolor=blue]{hyperref}
\usepackage[top=3cm, bottom=3cm, left=2.5cm , right=2.5cm]{geometry}
\usepackage{caption}
\usepackage{xifthen} 
\usepackage{wrapfig}
\usepackage[numbers, square]{natbib}
\usepackage{comment}
\usepackage{float}
\usepackage{ifthen}

\usepackage{caption}
\newtheorem{thm}{Theorem}

\newtheorem{lem}{Lemma}
\newtheorem{clm}{Claim}
\newtheorem{prop}{Proposition}
\newtheorem{cor}{Corollary}

\def\indi{\mathrm{\hspace{0.2em}l\hspace{-0.55em}1}}

\def\supp{{\rm Supp}}

\newcommand{\R}{\mathbb R}

\newcommand{\E}[2][]{\ensuremath{\mathbb{E}_{#1} \left[#2 \right]}}

\newcommand{\Prob}[2][]{\ensuremath{\mathbb{P}_{#1} \left(#2 \right)}}

\newcommand{\eps}{\varepsilon}

\newcommand{\St}[2]{\mathrm{st}_{#2}(#1)}
\newcommand{\Lk}[2]{\mathrm{lk}_{#2}(#1)}

\newcommand{\Ac}[1]{\mathcal{A}_{#1}}

\newcommand{\w}{\omega}

\newcommand{\cCd}{\mathcal {C}_{d-1}}

\newcommand{\dd}{\mathrm d}

\newcommand{\crit}{\mathrm cr}

\newcommand{\C}[3]{
\ifthenelse{\isempty{#3}}
{\ifthenelse{\isempty{#2}}
{{\mathcal K}_{#1}} 
{{\mathcal K}_{#1} \left(#2 \right)}
} 
{\ifthenelse{\isempty{#2}}
{{\mathcal K}_{#1}^{(#3)}} 
{{\mathcal K}^{(#3)} \left(#2 \right)}
}
}

\newboolean{finite_case} 
\setboolean{finite_case}{false} 

\newboolean{graphs} 
\setboolean{graphs}{false}   

\newboolean{wrrt} 
\setboolean{wrrt}{true}

\title{High-dimensional bootstrap processes  in evolving simplicial complexes}
\author{Nikolaos Fountoulakis\thanks{School of Mathematics, University of Birmingham, Edgbaston, Birmingham, UK. Supported by the EPSRC grant EP/P026729/1. \newline  E-mail: \texttt{\{n.fountoulakis, m.j.przykucki\}@bham.ac.uk}.}\ \ and Micha{\l} Przykucki\footnotemark[1]}

\begin{document}
\maketitle

\begin{abstract}
We study bootstrap percolation processes on random simplicial complexes of some fixed dimension $d \geq 3$. Starting from a single simplex of dimension $d$, we build our complex dynamically in the following fashion. We introduce new vertices one by one, all equipped with a random weight from a fixed distribution $\mu$. The newly arriving vertex selects an existing $(d-1)$-dimensional face at random, with probability proportional to some positive and symmetric function $f$ of the weights of its vertices, and attaches to it by forming a $d$-dimensional simplex. After a complex on $n$ vertices is constructed, we infect every vertex independently at random with some probability $p = p(n)$. Then, in consecutive rounds, we infect every healthy vertex the neighbourhood of which contains at least $r$ disjoint $(k-1)$-dimensional, fully infected faces. Using a reduction to the generalised P\'olya urn schemes, we determine the value of critical probability $p_c = p_c (n; \mu, f)$, such that if $p \gg p_c$ then, with probability tending to 1 as $n \to \infty$, the infection spreads to the whole vertex set of the complex, while if $p \ll p_c$ then the infection process stops with healthy vertices remaining in the complex.
\end{abstract}



\section{Introduction}
\normalem

The development of network science during the last 25 years has led to the development of a number 
of models whose aim is to describe the evolution of complex networks.  
This term broadly refers to autonomous networks that emerge as a result of collective human activity 
or natural processes. 
One of the most significant classes of models that allow us to understand the evolution of the real-world networks are the preferential attachment models. There,  nodes arrive in the network one by one and connect to some of the existing nodes with probability proportional to their popularity. The latter is usually quantified as an increasing function of their degree. This leads to a mathematical explanation of the \emph{rich-get-richer} phenomenon, as popular vertices with time tend to accumulate many more neighbours.

Preferential attachment models have been originally introduced in biology (see the work of Yule~\cite{Yule} on the evolution of species). Also, in linguistics Simon~\cite{Simon} introduced them as a tool to describe the statistics of language. Much later, Barab\'asi and Albert~\cite{Barabasi509} suggested that such mechanisms could be used as a model of complex networks. Their model was shortly afterwards defined rigorously by Bollob\'as, Riordan, Spencer and Tusn\'ady~\cite{bollobaspreferential}, who rigorously proved the emergence of the predicted effects in such graphs.

Among the phenomenona that we can observe in the preferential attachment random graph models, and which can also be observed in complex networks, are  degree distributions with tails following a power law (we call such networks \emph{scale free}), as well as the logarithmic or doubly logarithmic diameter, or the typical distance between two random vertices (we say that such networks form \emph{small} or~\emph{ultra-small} worlds, respectively). See, for example, \cite{vanderhofstad2016} for a detailed exposition of results on these models.

The simplest case of a preferential attachment model is that of the \emph{plane-oriented random recursive tree} introduced by Szyma\'{n}ski~\cite{szymanski87}. There, one builds a tree, where vertices arrive sequentially 
and become connected to one of the existing vertices selected with probability proportional to their degree. 
The limiting (as the number of vertices increases) degree distribution in this model was determined by Kuba and Panholzer~\cite{AllPan2007} in great detail.
We remark that if the neighbouring vertex is selected uniformly, one gets the well-known \emph{random recursive tree} model. 

Soon after the introduction of the Barab\'asi-Albert model, 
Bianconi and Barab\'asi~\cite{bianconibarabasi2001} introduced a generalisation of it. Here, every vertex 
is equipped with an independent random \emph{fitness} sampled from a given distribution, and then newly arriving vertices connect to an old vertex $v$ with probability proportional to the degree of $v$ \emph{multiplied} by its fitness. Depending on the properties of the fitness distribution, this generalisation leads to interesting new phenomena that do not occur in the Barab\'asi-Albert model. As observed in~\cite{bianconibarabasi2001}, and proved rigorously by Borgs et al.~\cite{Borgs2007} and later by Dereich and Ortgiese~\cite{dereich_ortgiese_2014}, \emph{condensation} can occur in the Bianconi-Barab\'asi model, which is the phenomenon when a vanishing fraction of vertices with very high fitness is incident to a positive fraction of the total number of edges in the graph.

Rudas, T\'oth and Valk\'o~\cite{rudas} introduced the following generalisation of the Bianconi-Barab\'asi model. Instead of having an attachment law based on a multiplicative fitness attached to every vertex, the newly arriving vertex connects to an old vertex $v$ with probability proportional to some increasing function of the degree of $v$. A more comprehensive overview of the other generalisations and variations of the preferential attachment model can be found in~\cite{vanderhofstad2016} and~\cite{bhamidi}.

\subsubsection*{Higher dimensional models: simplicial complexes}

In the context of \emph{simplicial complexes}, the above models result in  \emph{1-dimensional} structures, which only represent information about pairs of participating vertices. Models involving higher dimensional interactions (between larger sets of vertices), and thus encoding more complexity, are less well-studied. 

An \emph{(abstract) simplicial complex} $\C{}{V}{}$ on a ground set $V$ is a set of subsets of $V$ that is downwards closed: for any set $\sigma \in \C{}{V}{}$, if $\sigma' \subseteq \sigma$, then $\sigma' \in 
\C{}{}{}$ too. 
We call the elements of $V$ the vertices of $\C{}{V}{}$.
Any $\sigma \in \C{}{V}{}$ is called a \emph{face}, and we say that $\sigma$ has \emph{dimension} $s$ if 
$\sigma$ contains $s+1$ elements - we also call it an $s$-\emph{face} or an $s$-\emph{simplex}.  
For $s\in \mathbb N_0 := \{0,1,2, \ldots\}$, we let 
$\C{}{V}{s}$ denote the subset of $\C{}{V}{}$ consisting of all elements of dimension $s$.
The \emph{dimension} of $\C{}{V}{}$ is defined to be the maximum $s$ such that $\C{}{V}{s}$ is non-empty. 
\\\\
In this paper we study the so-called \emph{bootstrap percolation processes} on (random) simplicial complexes, which we now introduce.

\subsubsection*{A bootstrap process on simplicial complexes} 

Bootstrap processes on graphs were introduced by Chalupa, Leath and Reich in 1979~\cite{ar:CLR} as a 
simplified model that describes magnetization in crystallic structures. 
In the classical setting, the bootstrap process is an infection process on the vertex set of a graph $G=(V,E)$ 
which may be either finite or infinite. Each vertex in $V$ has one of two possible states: it is either \emph{infected} or
\emph{healthy}. A bootstrap process with \emph{infection threshold} $r$ is a process that evolves in rounds. 
Let $A_t\subseteq V$ be the subset of $V$ consisting of all those vertices that are infected after step $t$. 
Then $A_{t+1}$ consists of $A_t$ together with those vertices in $V \setminus A_t$ which 
have at least $r$ neighbours in $A_t$. The process $(A_t)_{t\geq 0}$ is non-decreasing and we say that 
it \emph{percolates}, if $\overline {A} = \cup_{t\geq 0} A_t = V$.   

In this paper, we shall consider a generalisation of the above process to simplicial complexes. 
Let $\C{}{V}{}$ be a simplicial complex of dimension $d$ with $V$ as its ground set. 
For  natural numbers $r$ and $k$ with $rk \leq d$, we will define the \emph{$(r,k)$-bootstrap process on $\C{}{V}{}$} 
as follows. 
In this setting too, the members of the ground set have two states: they are infected or healthy. 
We call a face $\sigma \in \C{}{V}{k}$ \emph{critical with respect to $A\subseteq V$} 
if all but one of its members are in $A$. 
If $A \subseteq V$ is a subset of active vertices, we let ${\mathcal F}_{\crit}^{(k)} (A)$ be the set of all critical $k$-faces with respect to $A$. We call these faces \emph{$A$-critical}. 

Let $A_t$ be the set 
of infected vertices after $t$ rounds. In round $t+1$ we set,  
\begin{align*}
A_{t+1} :=A_t \cup \{ v \in V \setminus A_t \ : \ & \exists \ \sigma_1,\ldots, \sigma_r \in  {\mathcal F}_{\crit}^{(k)} (A_t),  \\
& \{v\} = \sigma_i \cap \sigma_j \mbox{ for all } 1 \leq i < j \leq r  \}.
\end{align*}
In other words, if $v$ is healthy after $t$ rounds, it becomes infected in the $(t+1)$th round if it completes the infection of at least $r$ $A_t$-critical 
faces, which are pairwise disjoint apart from containing~$v$.

For a subset $A \subseteq V$ we again set $\overline {A} = \cup_{t=0}^{\infty} A_t$, for $A_0 = A$. Analogously to the bootstrap processes on graphs, we say that 
\emph{$A$ percolates}, if $\overline{A} = V$.

There are other models of high-dimensional bootstrap processes on hypergraphs - see for example~\cite{ar:KangKochMakai2015,ar:MorNo2018}. 

\subsubsection*{Our model: inhomogeneous dynamic simplicial complexes}

A model of randomly evolving simplicial complexes are Random Apollonian Networks, first introduced
by Andrade et al.~\cite{apollonianinitial} and Doye and Massen~\cite{DoyeMassen}, independently. This model (in dimension $d \geq 3$) proceeds recursively as follows. We begin with a $(d-1)$-dimensional complex, and at each step we select a $(d-1)$-face $\sigma$ uniformly at random. Then, a new vertex $v$ arrives and all $(d-1)$-faces containing $v$ and any of the $(d-1)$-element subsets of $\sigma$ are added to the complex, after which $\sigma$ is removed. It is easy to see that the probability of the degree of a vertex $w$ growing as a result of one of the faces containing $w$ being selected is proportional to its degree and this model gives rise to a \emph{preferential attachment mechanism}. 
Kolossv\'ary et al. \cite{apollonian_hungarians_1} and Frieze and Tsourakakis \cite{apollonian_frieze_2} determined the degree distribution of this model, showing that it gives rise to a power law with exponent $\frac{2d-3}{d-2} = 2 + \frac{1}{d-2}$. The same model has been considered under the name \emph{random stack-triangulations} by Albenque and Marckert in \cite{albenque2008}. 

In this paper, we consider a general model of \emph{inhomogeneous} simplicial complexes inspired by the above models.
The construction of our random simplicial complexes relies on a \emph{fitness function} $f: \R^d \to \R$ and a probability measure $\mu$ on $(0, 1]$ that has finite support.  To define our model, let us first introduce some notation. For $0 \leq s \leq d-1$, let ${\mathcal C}_s := \{(x_0, \ldots, x_{s}): 0 < x_0 \leq \ldots \leq x_{s} \leq 1 \}$. To each vertex $v$ of the simplicial complex, we associate a positive \emph{weight} $w_v$. For a $(d-1)$-face  $\sigma= \{i_0,\ldots, i_{d-1}\}$, by listing weights in non-decreasing order, weights induce a unique face \emph{type}
 $\omega(\sigma):= (w_{i_0},\ldots, w_{i_{d-1}}) \in \cCd$. 
 
Throughout this work, we make the following assumption:
\begin{itemize}
\item [\textbf{A1}.] $f$ is positive, symmetric and bounded.
\end{itemize}
For $x \in \cCd$ we call $f(x)$ the \emph{fitness} of $x$. As $f$ is symmetric, this notion extends in the obvious way to faces: we write  $f(\sigma)$ for $f(\omega(\sigma))$.
\\\\

We can now define a process of simplicial complexes 
$\left(\C{n}{}{}\right)_{n \geq 0}$ (of fixed dimension $d \geq 3$).
We will consider two different versions of the model: model \textbf{A} and model \textbf{B}. The dynamics of the process is as follows: 
 First, let $\C{0}{}{}$ be an arbitrary $d$-dimensional simplicial complex, with vertex set $V_0 \subseteq -\mathbb N$ and each vertex assigned a fixed positive weight. Set $\Ac{0} := \C{0}{}{d-1}$; we will call the faces in the consecutive sets $\Ac{t}$ \emph{active faces}.
Then, recursively for $n \geq 0$:
 \begin{itemize} 
 \item [(i)] Define the random empirical measure 
 \begin{align} \label{def:pin} \Pi_{n} = \sum_{\sigma \in \Ac{n}}\delta_{\w (\sigma)} \end{align} on $\cCd$ and the random finite discrete probability distribution on $\Ac{n}$:
 \[
 \hat{\Pi}_n = \frac{\sum_{\sigma \in \Ac{n}} f(\sigma) \delta_{\sigma}} {\int_{\cCd} f (x)  \Pi_n (\dd x)}.
 \]
\item [(ii)] Select a face $\sigma' \in \Ac{n}$ according to the distribution $\hat{\Pi}_n$. 

 \item [(iii)] The simplex $\sigma^* := \sigma' \cup \{n+1\}$ is added to $\C{n}{}{}$: letting $2^{\sigma^*}$ denote
 the powerset of $\sigma^*$, we set $\C{n+1}{}{} : = \C{n}{}{} \cup 2^{\sigma^*}$ and $V_{n+1} = V_n \cup \{n+1\}$ and say that face $\sigma'$ has been \emph{subdivided}. Vertex $n+1$ is assigned weight $w_{n+1}$ drawn according to the distribution $\mu$ independently of the history of the process.
 
 \item [(iv)] In Model \textbf{A}
  we set $\Ac{n+1} := \Ac{n} \cup (2^{\sigma^*})^{(d-1)}$; while in Model \textbf{B}, 
  we set $\Ac{n+1} := ( \Ac{n} \cup (2^{\sigma^*})^{d-1} )\setminus \{\sigma'\}$.
 \end{itemize}

 Note that by construction, only the active faces $\sigma \in \C{n}{}{d-1}$ with $\sigma \in \Ac{n}$ can be subdivided at time $n+1$.
\medskip 

In~\cite{ar:deg_distr2019}, the distribution of the degrees is considered in both models. There it is shown that 
when $d\geq 3$ its tails are bounded from below by a power law. 

We define
\[
Z(S_t) := \sum_{\sigma \in \mathcal{A}_{t}} f(\sigma) = \int_{\cCd} f (x)  \Pi_n (\dd x).
\]
We let $\Prob[\C{0}{}{}]{\cdot }$ and $\E[\C{0}{}{}]{\cdot}$ denote the probability measure of the process 
and the corresponding expectation 
operator when $\C{0}{}{}$ is the initial simplicial complex. 
For $n \in \mathbb{N}$, we let $\Prob[\C{0}{}{},n]{\cdot }$ denote the probability measure  of the above process after $n$ steps. We observe that the study of 
Model \textbf{B} is additionally motivated by the fact that it gives rise to a simplicial complex the 
geometric realisation of which is a topological manifold.
 
\subsubsection*{Some more notation}
Another notion that will be at the core of our results is that of the \emph{star} of a vertex. 
For $v\in V$, the \emph{star} of $v$ in a simplicial complex $\C{}{}{}$, which we denote by $\St{v}{\C{}{}{}}$, is the
 set of those $(d-1)$-faces which contain $v$. The $\emph{link}$ of a vertex $v$ in $\C{}{}{}$, denoted by $\Lk{v}{\C{}{}{}}$, is the simplicial complex obtained from $\St{v}{\C{}{}{}}$ by removing $v$ from the simplices in $\St{v}{\C{}{}{}}$, i.e., $\Lk{v}{\C{}{}{}} = \{ \sigma \setminus \{v\} : \sigma \in \St{v}{\C{}{}{}} \}$.
 
We shall say that an event occurs \emph{almost surely} (a.s.), if it occurs with probability 1. We further say that a sequence of events $\{A_n\}$ occurs \emph{with high probability} (w.h.p.) if the probability of $A_n$ tends to 1 as $n \to \infty$.

Finally, we point out that we use the interval notation for real numbers in the context of discrete time. 
So, for example, in that context $(a,b)$ denotes the set of natural numbers between $a$ and $b$. 
Closed or semi-open intervals have analogous meaning. 

\subsection{The $\star$-process} 
 Start with a $(d-1)$-face $\sigma_0$ (on the vertex set $\{-d, \ldots, -1\}$) and subdivide it with a vertex labelled $v_0$. Thereby,  
form the simplicial complex $S_0$ induced by the $(d-1)$-faces containing vertex $v_0$, taking the downwards 
closure of them. 
 We call $v_0$ the \emph{centre} of  $S_0$.
 Set $\mathcal{A}^{\star}_{0} := \St{v_0}{S_0}$. Then conditional on $S_t$: 
\begin{itemize}
  \item[(i)] Select a face $\sigma_t$ from $\mathcal{A}^{\star}_{t}$ with probability proportional to its fitness, and subdivide this with a new vertex $t+1$. 
 Form the simplicial complex $S_{t+1}$ discarding the only new face which does not contain $v_0$. 
  \item[(ii)]  In model \textbf{A}
  we set $\mathcal{A}^{\star}_{t+1} := \St{v_0}{S_{t+1}}$; while in model \textbf{B}, 
  we set $\mathcal{A}^{\star}_{t+1} := \St{v_0}{S_{t+1}} \setminus \{\sigma_s : 0 \leq s \leq t \}$.
\end{itemize}
If the centre of $S_0$, that is, $v_0$, has weight $x$, we call this process an \emph{$x\star$-process}. In fact, the Markovian nature of this process implies that its
limiting behaviour does not depend on $\sigma_0$. However, the value of $x$ determines its transition 
law and thereby the limiting behaviour. 
Similarly to the previous section, we define 
\[
Z^\star (S_t) := \sum_{\sigma \in \mathcal{A}^{\star}_{t}} f(\sigma).
\]

\subsection{Almost sure limits in models {\bf A} and {\bf B} and the $\star$-process}

Of particular importance are the limiting distributions of the types in these simplicial complexes. The following results also appear in~\cite{ar:deg_distr2019}.
\begin{prop}
\label{prop:MC}
There exists a probability measure $\pi$ on $\cCd$ such that almost surely $\Pi_n/|\Ac{n}| \to \pi$ weakly as $n\to \infty$.
In particular, if $\{Y_n\}_{n \geq 1}$ is the $\cCd$-valued Markov chain describing the type of the face chosen to be subdivided in the $n$th step, then in distribution, $Y_n \to Y_\infty$, as $n \to \infty$, for a $\cCd$-valued random variable $Y_\infty$ whose law is~$\pi$.

Furthermore, almost surely
\[
\frac{Z(\C{n}{}{})}{n} \to \lambda := \int f(w) \pi(\dd w) > 0,
\]
as $n\to \infty$.
\end{prop}
For the $x\star$-process one can show a similar result except that now the limit depends on $x$. 
\begin{prop} \label{prop:MC*} 
For $x\in(0,1]$ let $(S_n)_{n\geq 0}$ be an $x\star$-process. 
There exists a measure $\pi_x^\star$ on $\cCd$ such that  
almost surely
\[
\frac{Z^\star(S_n)}{n} \to  \lambda_x := \int f(y) \pi_x^\star (\dd y) > 0,
\]
as $n \to \infty$.

\end{prop}

We postpone the proof of the above propositions to Section~\ref{sec:polya_urns}.

\subsection{Main result: the critical density} 

We will consider the evolution of the bootstrap percolation process on $\C{n}{}{}$ in the case where the initially infected 
set $A_0$ is a \emph{binomial} subset of $V_n$, that is, vertices in $V_n$ become initially active with 
probability $p=p(n)$ independently. We denote the corresponding probability measure by $\mathbb A_{p,n}$

 Our aim is to find a critical value $p_c(n)$ for the infection probability, 
such that if $p \ll p_c(n)$, then no evolution occurs with high probability, whereas if $p \gg p_c (n)$, then 
the process percolates. 
We define $\lambda^\star= \max_{x\in \supp (\mu)}  \lambda_x$. As the main result of this paper, we prove the following theorem.

\begin{thm}
\label{thm:main}
Let $p_c(n) = n^{-\lambda^\star/(k \lambda)}$. Let $(A_t)_{t\in \mathbb{N} \cup \{0\}}$ be the $(r,k)$-bootstrap process on 
$\C{n}{}{}$ with $A_0$ sampled according to $\mathbb A_{p,n}$.
Let $\omega : \mathbb{N} \to \mathbb{R}$ such that $\omega (n)\to\infty$, as $n \to \infty$. 
\begin{enumerate}
\item For any $\delta>0$ there exists $n_0$ such that for all $n > n_0$ if $p=\omega (n) p_c$, then 
$$\left( \mathbb{P}_{\C{0}{}{},n} \otimes \mathbb A_{p,n} \right)( \{\overline A_0\} = V_n ) > 1-\delta.$$ 
\item If $\lambda^\star/ \lambda > 1/r$, then for $p =  p_c/\omega(n)$ we have 
$\left(\mathbb{P}_{\C{0}{}{},n}\otimes \mathbb A_{p,n}\right)( \{ A_1=A_0\}) > 1-\delta$.

If $\lambda^\star/ \lambda = 1/r$, then the same holds provided $p =  p_c (\log n)^{-1/rk}/\omega (n)$.
\end{enumerate}
\end{thm}
This critical behaviour was also proved for the classic (graph) bootstrap process on the Barab\'asi-Albert model by the first author and Abdullah~\cite{ar:AbdFountRSA} as well as by Ebrahimi et al.~\cite{ar:EGGS2017}. In a different context, such a phenomenon was also shown in inhomogeneous random graphs~\cite{ar:AmFount2014, ar:FKKM2018} as well as in random graphs 
on the hyperbolic plane~\cite{ar:CandFount2016}.  In~\cite{ar:FKKM2018} it is proved that in the class of inhomogeneous random graphs (Chung-Lu model) such a transition occurs essentially only when the 
degree distribution follows a power law with exponent less than 3. 

\subsubsection*{Sketch of proof}

We prove Theorem~\ref{thm:main} in Section~\ref{sec:mainResult}. The idea of the proof is as follows. In the supercritical case, we show that w.h.p. $\C{n}{}{}$ contains a face $\sigma$, all vertices of which have degrees high enough to contain at least $r$ disjoint $(k-1)$-dimensional faces in their link that are initially fully infected. This leads to a complete infection of the $d$ vertices of $\sigma$ at time $1$. Then, since $d \geq rk$, the infection spreads to the rest of the complex through the natural ``closing of $d$-dimensional simplices''.

The proof of the subcritical case uses a first moment argument. Namely, we show that with high probability no vertex has a degree high enough to have a chance to see at least $r$ disjoint $(k-1)$-dimensional faces in its link that are initially fully infected. Consequently, $A_1 = A_0$ and, in particular, we have no percolation.

One of the main assumptions in Theorem~\ref{thm:main} is the requirement that $\lambda^\star/ \lambda \geq 1/r$. This condition is critical for the first moment argument in the lower bound on $p_c$ to hold (however it plays no role in the proof of the upper bound, which holds regardless of the value of $\lambda^\star/ \lambda$). An important class of models for which Theorem~\ref{thm:main} does apply are the $d$-dimensional Random Apollonian Networks, which exhibit $\lambda^\star/ \lambda = \frac{d-2}{d-1}$. In Section~\ref{sec:examples} we present a much broader class of models, with a non-trivial weight distribution $\mu$ and fitness function $f$, for which the assumptions of Theorem~\ref{thm:main} hold.

\subsubsection*{The $d=2$ case}

So far, in our discussion we have assumed that $d \geq 3$. It is however interesting to ask what happens when $d=2$, i.e., when new vertices arrive and connect to existing edges by forming triangles. Here, however, there is a very significant qualitative difference between models \textbf{A} and \textbf{B}. In particular, in model \textbf{A} an equivalent of Theorem \ref{thm:main} also holds. On the other hand, in model \textbf{B}, when any edge can only be selected once in the whole evolution of the simplicial complex by an incoming vertex, the 1-skeleton of 
$\C{n}{}{}$, that is, the graph $(\C{n}{}{0}, \C{n}{}{1})$, is outerplanar, with the edges of the outer face constituting the set of active edges. Consequently, at any point in time, every vertex of the complex is incident to exactly two active edges and we have no preferential attachment mechanism. (It can be shown easily that for $d=2$, in model \textbf{B} we have $\lambda_x = 0$ for all weights $x$.) This results in the distributions of vertex degrees having exponential tails.

The condition that $rk \leq d = 2$ allows only two interesting bootstrap percolation processes: the $r=1, k=2$ process in which a vertex becomes infected if it completes an infected triangle, and the $r=2, k=1$ process in which a vertex becomes infected if it has at least two infected neighbours. In the first model, a set $A$ percolates if and only if it contains an infected edge. In the latter, $A$ needs to contain an infected edge or a pair of infected vertices at distance $2$. The concentration of degrees then leads to the critical probability of the order of $n^{-1/2}$.

\subsubsection*{Organisation of the paper}

The rest of this paper is structured as follows. In Section~\ref{sec:polya_urns} we recall the generalised P\'olya urn schemes, we define the reduction of the growth of random simplicial complexes to the evolution of P\'olya urns, and using this reduction we prove some results about the distribution of vertex degrees in our complexes. In Section~\ref{sec:mainResult} we prove Theorem~\ref{thm:main}, and in Section~\ref{sec:examples} we provide some examples of weight distributions and fitness functions for which Theorem~\ref{thm:main} applies. Finally, in Section~\ref{sec:further} we state some open problems and possible generalisations of our work.

\section{Reduction to generalised P\'olya urn schemes}
\label{sec:polya_urns}

In this section we recall some results about the generalised P\'olya urn schemes, and we show a coupling between the growth of our random complex and the evolution of the profile of an urn. Having established this coupling, we then exploit the theory of  P\'olya urns, in particular some results obtained by Athreya and Karlin, and by Janson, to obtain bounds on the tails of the distribution of vertex degrees, as well as on its moments. These bounds are then used in  Section~\ref{sec:mainResult} to prove Theorem~\ref{thm:main}.

Recall that in this paper, we shall consider the case where the vertex weight distribution $\mu$ is a point measure that is finitely supported in $(0,1]$: 
 $\mu = \sum_{i=1}^K \mu_i \delta_{x_i}$ for some $\mu_1, \ldots, \mu_K > 0$ adding up to one and 
 $x_1, \ldots, x_K \in (0,1]$. 
 
Then there are a finite number of face types, which by the fact that $f$ is symmetric depend on the multiset of weights of the vertices contained in a face. We can view these types as \emph{colours} in an urn scheme with a finite number 
of colours. More specifically, a face $\sigma$ has colour $\w (\sigma)$ which belongs to the set of multisets of size 
$d$ whose elements are in $\{x_1,\ldots, x_K\}$. (Note that there are $q:={d+K-1\choose d}$ such multisets.) 
Furthermore, the multiset $\w(\sigma)$ has fitness $f(\omega(\sigma))>0$. 

The evolution of the simplicial complex in our model can be seen as an
urn scheme which contains balls, each having a type/colour from a set of $q$ types/colours. 
Each type has a positive weight. 
Model {\bf A} corresponds to the following replacement scheme. At each step, a ball is selected with probability proportional 
to the weight of its type. Then, $d$ balls are added to the urn, where the types of the balls are random and are determined as follows. The  weight of the newly arrived vertex 
is sampled according to $\mu$. This together with each one of the $d$ $(d-1)$-subsets of the sampled face 
creates $d$ multisets of weights. These are the new balls that are added. 
Model {\bf B} is similar, but the ball that has been selected is removed from the urn. 

Such general urn schemes embedded in \emph{continuous} time were considered by Athreya and Karlin~\cite{ar:AthrKarlin68}. 
Let $X(t) = (X_1(t),\ldots, X_q (t))$ be the composition vector of an urn containing balls with 
$q$ possible types. For each $i=1,\ldots, q$, a  ball of type $i$ has activity/fitness $a_i>0$. 
A ball of type $i$ dies with rate $a_i$ and at the moment of its death it produces balls 
of all types according to a given distribution. 

Let $M_{ij}(t) = \E{X_j (t)  \ | \ X_r (0) = \delta_{ir}, \ r=1,\ldots, q}$; in other words, $M_{ij}(t)$ 
is the expected number of balls of type $j$ at time $t$ given that the initial configuration consists 
of a single ball of type $i$. Furthermore, $M(t) = \exp (At)$, where $A$ is 
the infinitesimal generator of the process: $A_{ij} = a_i B_{ij}$, where $B_{ij}$ is the expected 
number of balls of type $j$ a ball of type $i$ produces upon its death. (In model {\bf B}, when $j=i$, we subtract 1 
to account for the ball of type $i$ that has died.)

The basic assumption on $M$ is the \emph{irreducibility assumption}: there exists $t_0$ such that 
for all $i,j=1,\ldots q$ we have $M_{ij}(t_0) >0$. 
If this is the case, then $A$ has a unique eigenvalue $\lambda$ of  maximum real part, which is real and 
there exist left and right eigenvectors $u,v$, respectively, with positive entries such that 
$Av = \lambda v$ and $u^T A = \lambda u^T$. We assume that they are normalised so that $u \cdot v =1$, where 
$\cdot$ denotes the usual dot product of two vectors in $\mathbb{R}^q$. 
  
Let $\tau_n$ be the stopping time which is the time when the $n$th split occurs, 
and let $X_n = X(\tau_n)$; this is the composition vector of the urn after $n$ balls have split.  
 
 The next theorem follows from Proposition 2 in~\cite{ar:AthrKarlin68} together with Theorem 5 
 from~\cite{ar:AthrKarlin67}.  
\begin{thm}  \label{thm:as_convergence}
Let $\rho = \nu (u_1,\ldots, u_q)^T$, where $\nu = \lambda \left(\sum_{i=1}^q a_i u_i\right)^{-1}$. 
Then almost surely 
\begin{align} \frac{X_n}{n} \to \rho  \ \mbox{as $n \to \infty$}. \end{align} 
\end{thm}
\begin{proof}
By Proposition 2 in~\cite{ar:AthrKarlin68} we know that there is an almost surely finite random variable $W$ such that $X_n e^{-\lambda \tau_n} \to Wu$ as $n \to \infty$. Theorem 5 from~\cite{ar:AthrKarlin67} says that we simultaneously have $n \nu e^{-\lambda \tau_n} \to W$ as $n \to \infty$. Hence, as
$X_n e^{-\lambda \tau_n} = \tfrac{X_n}{\nu n} n \nu e^{-\lambda \tau_n}$, we see that $\tfrac{X_n}{\nu n} \to u$, and the theorem follows.
\end{proof}

Propositions~\ref{prop:MC} and~\ref{prop:MC*} are straightforward consequences of the above theorem.

\begin{proof}[Proof of Proposition~\ref{prop:MC}]
Consider model \textbf{A}; the proof in model \textbf{B} is analogous. Here, we have $|\Ac{n}| = dn+d+1$, and so $\tfrac{\Pi_n}{|\Ac{n}|} = \tfrac{\Pi_n}{dn+d+1}$. Since $\Pi_n(\omega(\sigma))$ is equal to the number of active faces of type $\omega(\sigma)$, we can associate it with the number of balls of colour $\omega(\sigma)$, denoted $X_{\omega(\sigma)}(n)$, in the appropriate generalised P\'olya urn scheme. By Theorem~\ref{thm:as_convergence} we have
\[
\frac{X_n}{n} \to  \frac{\lambda (u_1,\ldots, u_q)^T}{\sum_{i=1}^q a_i u_i},
\]
therefore, slightly abusing the notation, we obtain that
\[
\tfrac{\Pi_n}{dn+d+1} \to \pi := \frac{\lambda \sum_{i=1}^q u_i \delta_{\omega(i)}}{d \sum_{i=1}^q a_i u_i}.
\]

The almost sure convergence of $\tfrac{\Pi_n}{|\Ac{n}|}$ implies the existence of the almost sure limit of $\frac{Z(\C{n}{}{})}{n}$. However, with a bit more work we will be able to identify this limit as the maximum real eigenvalue $\lambda$ of the infinitesimal generator matrix $A$ of the appropriate generalised P\'olya urn scheme.

By the fact that $u$ is a left eigenvector of $A$ with eigenvalue $\lambda$ and with $\indi$ being the 
all-1s vector in $\mathbb{R}^q$, we can see that
\[
\sum_{i=1}^q u_i = u^T \indi = \frac{1}{\lambda} u^T A \indi.
\]
Any row $i$ in the matrix $A$ sums to the expected number of balls that are created when a ball of colour $i$ dies, multiplied by the activity $a_i$ of balls of that colour. In model \textbf{A} the number of balls created is always equal to $d$, therefore we have
\[
\frac{1}{\lambda} u^T A \indi = \frac{1}{\lambda} u^T (a_1 d, \ldots, a_q d)^T = \frac{d}{\lambda} \left ( \sum_{i=1}^q a_i u_i \right ),
\]
and $\pi$, the limit of $\Pi_n/|\Ac{n}|$ is indeed a probability distribution. Finally, recalling that the fitness $f(\omega(\sigma))$ of a face $\sigma$ is equivalent to the activity $a_i$ of the appropriately coloured ball in the generalised P\'olya urn scheme, we get
\[
\frac{Z(\C{n}{}{})}{n} = d \frac{\int_{\cCd} f (x)  \Pi_n (\dd x)}{d n} \to d \int_{\cCd} f (x) \pi(dx) = d \sum_{i=1}^q a_i \frac{\lambda u_i}{d \sum_{j=1}^q a_j u_j} = \lambda.
\]
\end{proof}
The proof of Proposition~\ref{prop:MC*} in analogous.

\subsection{The distribution of degrees}

In this section, we will give estimates on the tails of the distribution of the degree of vertex $i_0$ at time $n$, as well as on the expected rate of growth of the $r$th power of the degree. 
To this end, we shall use the continuous time embedding of the generalised P\'olya process.

We will consider the generalised P\'olya urn which corresponds to the $\star$-process rooted at $i_0$. 
Let $Y_t^{(i_0)}:=X_{t+\tau_{i_0}}^{(i_0)}=(X_1^{(i_0)}(t+\tau_{i_0}),\ldots, X_q^{(i_0)}(t+\tau_{i_0}))$ denote the composition vector of the generalised P\'olya urn representing the star of $i_0$. We will use a result of Janson~\cite{ar:Janson2004} (Lemma 10.2, p. 232) which bounds the expectation of the $r$th norm of $Y_t^{(i_0)}$. 
\begin{lem} \label{lem:rth_norm} If vertex $i_0$ has fitness $x$, then 
for all $t\geq 0$ we have 
\[
\|Y_t^{(i_0)}\|_r = \left ( \E{  \sum_{i=1}^q \left( X_i^{(i_0)}(t+\tau_{i_0})  \right )^r } \right )^{1/r} \leq C e^{\lambda_x t}.
\]
\end{lem}
Since $(\sum_{i=1}^q x_i)^r \leq q^r \sum_{i=1}^q x_i^r$ where $x_i \geq 0$ for all $i$,  
we obtain the following bound. 
\begin{cor} With $i_0$ as above, for all $t \geq 0$ we have
\[
\E{ \left( \sum_{i=1}^q X_i^{(i_0)}(t+\tau_{i_0}) \right)^r} \leq q^r C e^{r \lambda_x t}. 
\]
\end{cor}
But note that $ \sum_{i=1}^q X_i^{(i_0)}(t+\tau_{i_0})$ is equal to the number of vertices that have 
become adjacent to $i_0$ by time $t+\tau_{i_0}$. Note also that the degree 
of $i_0$ by time $t$ is in fact equal to $d + \sum_{i=1}^q X_i^{(i_0)}(t+\tau_{i_0})$, as $i_0$ has degree $d$ the moment it is generated.  
So if we denote by $D_{t}(i_0)$ the degree of $i_0$ at time $t+ \tau_{i_0}$, we then deduce that for all 
$t\geq 0$
\begin{align} \label{eq:rth_moment} 
\E { D_{t}^r (i_0)} & = \E {\left (d + \sum_{i=1}^q X_i^{(i_0)}(t+\tau_{i_0})\right )^r} \nonumber \\
 & = \E{ \sum_{i=0}^r \binom{r}{i} d^{r-i} \left( \sum_{i=1}^q X_i^{(i_0)}(t+\tau_{i_0}) \right)^i} \nonumber \\
 & \leq 2^r d^r \E{ \left( \sum_{i=1}^q X_i^{(i_0)}(t+\tau_{i_0}) \right)^r} \nonumber \\
 & \leq (2dq)^r C e^{r\lambda_x t},
\end{align}
where $C$ is as in Lemma~\ref{lem:rth_norm}.  

Furthermore, we will make use of the following result of Athreya and Karlin~(see Theorem 6 in~\cite{ar:AthrKarlin67}) regarding 
the times of arrival of the vertices.
\begin{thm}
\label{thm:splitting_times_convergence}
The value of $\tau_n-\log n / \lambda$ converges almost surely to a finite valued random variable.
\end{thm}

From Theorem~\ref{thm:splitting_times_convergence} we obtain the following immediate corollary.
\begin{cor}
\label{cor:splitting_times_conc}
For any $\delta >0$ and for any $\eps> 0$ there exists $n_0 \in \mathbb {N}$ such that 
\begin{equation} \label{eq:splitting_times_conc}
\Prob[\C{0}{}{}] {\left| \tau_n - \tau_i - \frac{1}{\lambda} \log \frac{n}{i} \right|< \eps, i=n_0,\ldots, n} > 1-\delta. 
\end{equation}
\end{cor}
We denote the event in Corrollary~\ref{cor:splitting_times_conc} by $\mathcal{E}_{\eps,n_0,n}$.

Theorem~\ref{thm:splitting_times_convergence} also allows us to control the time of the $n$th split.
\begin{lem}
\label{lem:split_time_distr}
For any $\delta > 0$ there exists some positive constant $C$ and $n_1 \in \mathbb {N}$ such that for $n \geq n_1$ we have
\begin{equation} \label{eq:split_time_distr}
\Prob[\C{0}{}{}] { \left | \tau_n - \frac{\log n}{\lambda} \right | > C } \leq \delta. 
\end{equation}
\end{lem}
\begin{proof}
Let us show that by taking $C$ sufficiently large the probability that $\tau_n > \frac{\log n}{\lambda} + C$ can be made arbitrarily small for large $n$. The opposite direction then follows analogously.

Let $W$ be the random variable that is the a.s.-limit of $\tau_n-\log n / \lambda$, which is finite almost surely by Theorem~\ref{thm:splitting_times_convergence}, and let $K$ be such that
 \[
  \Prob[\C{0}{}{}] { W > K } \leq \delta / 2.
 \]
 Let $n_1$ be such that for $n \geq n_1$ we have
 \[
  \Prob[\C{0}{}{}] { \left | W - \left (\tau_n- \frac{\log n}{\lambda} \right ) \right | > K } \leq \delta/2.
 \]
 Then by the triangle inequality for all $n \geq n_1$  we have that
\[
  \Prob[\C{0}{}{}] { \tau_n- \frac{\log n}{\lambda} > 2K } \leq \Prob[\C{0}{}{}] { W > K } + \Prob[\C{0}{}{}] { \left | W - \left (\tau_n- \frac{\log n}{\lambda} \right ) \right | > K } \leq \delta.
\]
 
\end{proof}

Corollary~\ref{cor:splitting_times_conc} and Lemma~\ref{lem:split_time_distr} allow us to deduce a bound on the lower tail of the distribution of the degree of a vertex.

\begin{lem}
\label{lem:degree_lower_bound}
Let $i_0$ be a vertex of weight $x$. Then for any $\delta > 0$ there exists some positive constant $c$ such that for $n$ large enough we have
\begin{equation}
\label{eq:degree_lower_bound}
\Prob[\C{0}{}{}] { D_{n} (i_0) > c \left( \frac{n}{i_0} \right )^{\lambda_x / \lambda} } > 1-\delta. 
\end{equation}
\end{lem}
\begin{proof}
 Let $n_2 = \max \{n_0,n_1\}$, where $n_0$ and $n_1$ are as in Corollary~\ref{cor:splitting_times_conc} and Lemma~\ref{lem:split_time_distr} respectively. By Corollary~\ref{cor:splitting_times_conc} we know that there is some $\eps > 0$ such that with probability at least $1-\delta/2$ we have
 \[
 \tau_n - \tau_{\max \{n_2, i_0 \}} > \frac{1}{\lambda} \log \frac{n}{\max \{n_2, i_0 \}} - \eps.
 \]
 Let $\tau^\star_j$ denote the moment of the $j$th split in the continuous time embedding of the generalised P\'olya urn scheme process we can couple with the companion $x\star$-process associated with vertex~$i_0$. By Lemma~\ref{lem:split_time_distr} we obtain that for $\alpha > 0$, if $n$ is large enough then with probability at least $1-\delta/2$ we have
 \begin{align*}
 \tau^\star_{\alpha (n/\max \{n_2, i_0 \})^{\lambda_x/\lambda}} & \leq \frac{\log \left ( \alpha \left ( \frac{n}{\max \{n_2, i_0 \})} \right )^{\lambda_x/\lambda} \right)}{\lambda_x} + C \\
 & = \frac{1}{\lambda} \log \frac{n}{\max \{n_2, i_0 \}} + \frac{1}{\lambda_x}\log \alpha + C.
 \end{align*}
 Hence, if $\alpha > 0$ is small enough then by the union bound with probability at least $1-\delta$ we have
 \begin{align*}
 \tau^\star_{\alpha (n/\max \{n_2, i_0 \})^{\lambda_x/\lambda}} & \leq \frac{1}{\lambda} \log \frac{n}{\max \{n_2, i_0 \}} + \frac{1}{\lambda_x}\log \alpha + C \\
 & \leq \frac{1}{\lambda} \log \frac{n}{\max \{n_2, i_0 \}} - \eps \\
 & \leq \tau_n - \tau_{\max \{n_2, i_0 \}}.
 \end{align*}
 Hence, in the continuous time embedding of the urn scheme, the time from the $\max \{n_2, i_0 \}$th split to the $n$th split is longer than the time from the birth of the vertex $i_0$ to its degree reaching value at least $\alpha (n/\max \{n_2, i_0 \})^{\lambda_x/\lambda}$.
 
 Recalling that $n_2$ is a constant we then see that there is some constant $c$ such that
 \[
 \Prob[\C{0}{}{}] { D_{n} (i_0) > c \left( \frac{n}{i_0} \right )^{\lambda_x / \lambda} } > 1-\delta 
 \]
 as desired.
 \end{proof}
 
Analogously to Lemma \ref{lem:degree_lower_bound}, we can prove the following bound on the upper tail of the distribution of $D_{n} (i)$.
\begin{lem}
\label{lem:degree_upper_bound}
Let $i_0$ be a vertex of weight $x$. Then for any $\delta > 0$ there exists some positive constant $C$ such that for $n$ large enough we have
\begin{equation}
\label{eq:degree_upper_bound}
\Prob[\C{0}{}{}] { D_{n} (i_0) < C \left( \frac{n}{i_0} \right )^{\lambda_x / \lambda} } > 1-\delta. 
\end{equation}
\end{lem}
\begin{flushright}$\Box$\end{flushright}

We will use Corollary~\ref{cor:splitting_times_conc} in order to bound the $r$th moment of the degree of vertex $i_0$ at 
the time the $n$th vertex arrives, that is, at $\tau_n$. 
In particular, for $n_0 \leq i_0 \leq n$ we have the following inequality: 
\begin{align*}
\E[\C{0}{}{}]{\mathbf 1 (\mathcal{E}_{\eps, n_0,n}) D_{\tau_n - \tau_{i_0}}^r (i_0)} & \leq 
\E[\C{0}{}{}]{\mathbf 1 (\mathcal{E}_{\eps,n_0,n}) D_{\frac{1}{\lambda} \ln \left( \frac{n}{i_0}\right)+\eps}^r (i_0)}  \\
&\leq \E[\C{0}{}{}]{D_{\frac{1}{\lambda} \ln \left( \frac{n}{i_0}\right)+\eps}^r (i_0)} \\
& \stackrel{\eqref{eq:rth_moment}} {\leq} (2dq)^r C e^{r\frac{\lambda_x}{\lambda} \ln \left( \frac{n}{i_0}\right) + r \lambda_x \eps} \\
&\leq C^* \cdot \left(\frac{n}{i_0}\right)^{\frac{r\lambda_x}{\lambda}},
\end{align*} 
where $C^*=(2dq)^r C e^{r\lambda_x \eps}$.
Thus, recalling that $\lambda^\star = \max_{x \in \supp (\mu)} \lambda_x$, we deduce that 
for all $n_0 \leq i_0 \leq n$ we have 
\begin{equation} \label{eq:rthmom_degree}
\E[\C{0}{}{}]{\mathbf 1 (\mathcal{E}_{\eps, n_0,n}) D_{\tau_n - \tau_{i_0}}^r (i_0)} \leq 
C^* \cdot \left(\frac{n}{i_0}\right)^{\frac{r\lambda^\star}{\lambda}}.
\end{equation}

We will use this bound later to complete a first moment argument towards the proof of the subcritical case in Theorem~\ref{thm:main}. 

\section{Proof of Theorem~\ref{thm:main}}
\label{sec:mainResult}

\subsection{Supercritical case}

The idea of the proof of the supercritical case is as follows. First, we wait until a $(d-1)$-dimensional face $\sigma$ consisting of $d$ vertices of weight $x_{max}$, with $\lambda_{x_{max}} = \lambda^\star$, is created. Then, we observe the companion process associated with a fixed vertex $v_0 \in \sigma$. By Lemma \ref{lem:degree_lower_bound}, with probability close to $1$ the degree of $v_0$ by time $n/d$ is at least $c_1 \cdot n^{\lambda^\star / \lambda}$ for some constant $c_1 > 0$, implying that at time $n/d$ the star of vertex $v_0$ contains at least $c_2 \cdot n^{\lambda^\star / \lambda}$ active $(d-1)$-dimensional faces. Next, we look at the faces in the star of $v_0$ that were active at time $n/d$, and we show that some positive proportion of these will be subdivided at least once in the time interval $(n/d, 2n/d]$, with every such face $\sigma_0$ upon subdivision producing an active face having at most $d-1$ vertices (including $v_0$) in common with $\sigma_0$. We then show that a positive proportion of these faces will be subdivided in the time interval $(2n/d, 3n/d]$, producing an active face having at most $d-2$ vertices (including $v_0$) in common with $\sigma_0$. We repeat this argument $d-1$ times, consequently producing at least $c \cdot n^{\lambda^\star / \lambda}$ faces at time $n$ that all contain $v_0$, and are otherwise pairwise disjoint. Then, a simple concentration argument shows that if $p \gg n^{-\lambda^\star/(k \lambda)}$ then $v_0$ will be contained in at least $r$ $k$-dimensional (and otherwise disjoint) faces that are critical with respect to the initially infected set, with $v_0$ being the only initially uninfected vertex in these faces. Hence $v_0$ becomes infected in one step, and by the union bound with probability close to $1$ the same applies to the other vertices of $\sigma$. Once the $(d-1)$-dimensional face $\sigma$ is infected, by the fact that $d \geq r k$, the infection spreads to every vertex of the complex.

\begin{proof}[Proof of the upper bound in Theorem~\ref{thm:main}]

Let $\rho=(\rho_1,\ldots, \rho_q)$ be as in Theorem~\ref{thm:as_convergence}, let $\delta > 0$, and let
\[
\eps = \min_{1 \leq i \leq q} \rho_i/2 > 0.
\]
Let $X_\ell$ denote the composition vector of the generalised P\'olya urn that corresponds to $\C{\ell}{}{}$, for any 
$\ell \in \mathbb{N}_0$. 
By the almost sure convergence of $X_\ell / \ell$ to $\rho$, as $\ell \to \infty$, we know that for any $\delta > 0$ there is some $M_{\eps,\delta}$ such that for all $\ell \geq M_{\varepsilon,\delta}$ with probability at least $1-\delta/2$ we have $\| X_\ell/\ell-\rho \|_\infty \leq \eps$, and consequently $\Ac{\ell}$ contains at least one face consisting only of vertices of weight $x_{max}$. Hence, let $\sigma \in \Ac{M_{\eps,\delta}}$ be such a face.

Let $v_0 \in \sigma$ be fixed, and let $\mathcal{A}^{\star}_{\ell}$ denote the set of active faces in $\St{v_0}{\C{\ell}{}{}}$, the star of $v_0$ after step $\ell$. By Lemma~\ref{lem:degree_lower_bound}, there exists a $\beta > 0$ such that with probability at least $1-\delta/(2d)$ we have $D_{n/d} (v_0) > \beta ( n/v_0 )^{\lambda^\star / \lambda}$, which implies that for some constant $\alpha_0 > 0$ at time $n/d$ vertex $v_0$ is in at least 
$\alpha_0 n^{\lambda^\star / \lambda}$ different $(d-1)$-dimensional faces contained in $\mathcal{A}^{\star}_{n/d}$. We would like to claim that with high probability there are at least $r$ pairwise disjoint (excluding $v_0$) such faces, in which all vertices other than $v_0$ are initially infected; however, some neighbours of $v_0$ could be contained in a large number of faces in $\St{v_0}{\C{n/d}{}{}}$, making it harder to find the desired structure.

The total fitness of the faces in $\mathcal{A}^{\star}_{n/d}$ is therefore at least $\alpha_0 n^{\lambda^\star / \lambda} f'$, where $f'$ is the minimum fitness of a face containing a vertex of weight $x_{max}$ (in our case $v_0$). We will call the faces  in $\mathcal{A}^{\star}_{n/d}$ \emph{$0$-new}; let $\mathcal{B}_{0}$ denote the set of $0$-new faces. We claim that in the time interval $(n/d, 2n/d]$ at least some constant proportion of the $0$-new faces becomes subdivided at least once.

\begin{clm} \label{clm:subdivisions}
Given that $|\mathcal{B}_0| \geq \alpha_0 n^{\lambda^\star /\lambda}$, there is some $\alpha_1 > 0$ such that w.h.p. at least 
$\alpha_1 n^{\lambda^\star /\lambda}$ members of $\mathcal{B}_0$ are subdivided during the interval 
$(n/d,2n/d]$.
\end{clm}
\begin{proof}[Proof of Claim~\ref{clm:subdivisions}]
Indeed, until at least $\alpha_0 n^{\lambda^\star / \lambda} / 2$ of the $0$-new faces become subdivided at least once, the total fitness of such faces that have not been subdivided at all is at least $\alpha_0 n^{\lambda^\star / \lambda} f' / 2$. On the other hand, recalling that $|\Ac{\ell}| \leq d\ell+d+1 \leq (d+1)\ell < 2d \ell$, for any $n/d \leq \ell < 2n/d$ the total fitness of the faces in $\mathcal{A}_{\ell}$ is at most $4n \tilde{f}$, where $\tilde{f} = \max_{x \in \cCd} f(x)$.

For $j\geq 1$, let $\tau_j > n/d$ be the stopping time of the $j$th subdivision of a $0$-new face after step $n/d$. 
(Also, set $\tau_0 := \lfloor n/d \rfloor$.) 
Set $E_j = \tau_j - \tau_{j-1}$ to be the number of steps between the $(j-1)$th subdivision of a $0$-new face and 
the $j$th subdivision of a $0$-new face within the interval $(n/d,2n/d]$ - we call this the duration of the $j$th epoch. 
The above observation implies that for $j\leq \alpha_0 n^{\lambda^\star / \lambda} / 2$, 
each subdivision that takes place inside the $j$th epoch is a subdivision of a $0$-new face with probability 
at least $p:=\left(\alpha_0 n^{\lambda^\star / \lambda} f' / 2\right)/ (4n \tilde{f})$ uniformly over the history of the 
process up to that point. 
Thus, for such $j$, we have $\Prob[\C{n/d}{}{}]{\tau_j - \tau_{j-1} \geq k} \leq (1-p)^k$.
We conclude that the random variable $E_j$ is stochastically bounded from above by 
a geometrically distributed random 
variable with parameter equal to $p$, which we 
denote by $Y_j$. Furthermore, the collection $Y_j$, $j\leq  \alpha_0 n^{\lambda^\star / \lambda} / 2$, form 
an independent family.  

Now, let $\alpha_1 < \frac{\alpha_0}{8}\frac{f'}{\tilde{f}}$ and assume that 
$\C{n/d}{}{}$ has at least $\alpha_0 n^{\lambda^\star /\lambda}$ $0$-new faces in the star of $v_0$. 
The stochastic domination yields  
$\Prob[\C{n/d}{}{}] {\sum_{j\leq \alpha_1 n^{\lambda^\star / \lambda}} E_j > n/d} \leq 
\mathbb{P} \left( \sum_{j\leq \alpha_1 n^{\lambda^\star / \lambda}} Y_j > n/d\right)$.
But $\mathbb{E} \left[\sum_{j\leq \alpha_1 n^{\lambda^\star / \lambda}} Y_j  \right] \leq \frac{8\alpha_1}{\alpha_0} 
\frac{\tilde{f}}{f'} n$. 
If $\alpha_1$ is small enough so that $\frac{8\alpha_1}{\alpha_0} 
\frac{\tilde{f}}{f'}  < 1/d$, Chebyschev's inequality implies that the above probability is $o(1)$. 
In other words, with probability $1-o(1)$, at least $\alpha_1 n^{\lambda^\star / \lambda}$ $0$-new faces
are subdivided during the interval $(n/d, 2n/d]$. 
\end{proof}

Hence, given any $\sigma^0 \in \mathcal{B}_{0}$, label the vertices of $\sigma^0 \setminus \{v_0\}$ as $v_1, \ldots, v_{d-1}$ in an arbitrary way (vertices that belong to multiple faces in $\mathcal{B}_{0}$ receive one label from each of these faces). If $\sigma^0$ is subdivided for the first time in the time interval $(n/d, 2n/d]$, label the vertex that subdivides $\sigma^0$ with $w_1$, and let $\sigma^1 = (v_0,w_1,v_2,\ldots,v_{d-1})$; we say that $\sigma^1$ is $1$-new, and we denote the set of $1$-new faces with $\mathcal{B}_{1}$. 
If $\sigma^1$ is subdivided during the time interval $(n/d, 2n/d]$, we remove $\sigma^1$ from $\mathcal{B}_{1}$ and we simply relocate the labels: we remove the label from the vertex $w_1$, we place it on the vertex that subdivided the face $\sigma^1$. We then remove the label from the face $\sigma^1$, and finally we again define $\sigma^1 = (v_0,w_1,v_2,\ldots,v_{d-1})$, which we declare $1$-new and place it in $\mathcal{B}_{1}$. We repeat this procedure whenever a face in $\mathcal{B}_{1}$ becomes subdivided in the time interval $(n/d, 2n/d]$. Note that $\sigma^1$ is the unique $1$-new face that contains the vertex $w_1$.

As in the above claim, one can show that there is some constant $\alpha_2 > 0$ such that w.h.p. at least $\alpha_2 n^{\lambda^\star / \lambda}$ faces in $\mathcal{B}_{1}$ are subdivided at least once during the time interval 
$(2n/d, 3n/d]$. When a $1$-new face $\sigma^1 = (v_0,w_1,v_2,\ldots, v_{d-1}) \in \mathcal{B}_{1}$ becomes subdivided for the first time in that interval, we label the vertex that subdivided it $w_2$, we define $\sigma^2 = (v_0,w_1,w_2,\sigma_3,\ldots,\sigma_{d-1})$, and we declare $\sigma^2$ to be a $2$-new face; we denote the set of $2$-new faces with $\mathcal{B}_{2}$. We then proceed as in the case of the $1$-new faces - if a $2$-new face $\sigma^2 \in \mathcal{B}_{2}$ is subdivided in the time interval $(2n/d, 3n/d]$, we move the label $w_2$ to the new vertex and we again redefine $\sigma^2$. Consequently, at time $3n/d$ we have at least $\alpha_2 n^{\lambda^\star / \lambda}$ faces in $\mathcal{B}_{2}$ with the additional property that any such face $\sigma^2 = (v_0,w_1,w_2,\sigma_3,\ldots,\sigma_{d-1})$ is the unique $2$-new face containing any of the vertices 
$w_1,w_2$.

For $3 \leq i \leq d-1$, we proceed in a similar fashion in the time intervals $(in/d, (i+1)n/d]$, w.h.p. obtaining the set $\mathcal{B}_{i}$ of at least $\alpha_i n^{\lambda^\star / \lambda}$ $i$-new faces, with the property that each such face contains at least $i$ vertices that do not belong to any other $i$-new face. Taking $i = d-1$ we see that for some $\alpha_{d-1} > 0$, at time $n$ vertex $v_0$ belongs to at least $\alpha_{d-1} n^{\lambda^\star / \lambda}$ faces in $\mathcal{B}_{d-1}$ that are disjoint apart from containing $v_0$. Since in the supercritical regime we have $p \gg n^{-\lambda^\star/(k \lambda)}$, by Chernoff bounds we immediately obtain the fact that with high probability $v_0$ is contained in at least $r$ faces in $\mathcal{B}_{d-1}$ that are critical with respect to the initially infected set, with $v_0$ being the only initially uninfected vertex.

We continue the proof of the upper bound by taking the union bound over the vertices in the initially selected face $\sigma$ containing only vertices of weight $x_{max}$. Hence, for any $\delta > 0$, if $p \gg n^{-\lambda^\star/(k \lambda)}$, then with probability at least $1-\delta-o(1)$ after the first step of the bootstrap process the infected set contains a fully infected $(d-1)$-dimensional face. Since we have $d \geq rk$, in consecutive steps of the bootstrap process infection spreads from that face to the rest of the complex. Consequently, since $\delta > 0$ was arbitrary, percolation occurs with probability $1-o(1)$.

\end{proof}

\subsection{Subcritical case}

Next, we show that under the assumptions in Part 2 of Theorem \ref{thm:main} the initially infected set is stable with high probability, i.e., the infection does not spread at all, and in particular there is no percolation. Recalling that $A_t$ is the set of vertices that are infected at time $t$, this is equivalent to showing that $A_1 \setminus A_0 = \emptyset$.

\begin{proof}[Proof of the lower bound in Theorem~\ref{thm:main}]
Since $p = o(1)$, w.h.p. after the initial infection is seeded there are some initially healthy vertices in the complex, which need to be infected in the bootstrap percolation process. Let $v$ be a vertex in our complex and assume that $v$ is initially healthy but becomes infected in the first step of the process. Hence, the link $\Lk{v}{\C{n}{}{}}$ of $v$ contains $r$ pairwise disjoint $(k-1)$-dimensional fully infected faces. Crucially, by the definition of our complex, any such face must be contained fully in one of the $d$-dimensional simplices containing $v$, which is trivially at most the degree of $v$.

Hence, conditioned on the degree of $v$ at time $n$, the probability that $v$ becomes infected in the first step of the process is at most
\[
  \left (D_n(i)  \binom{d}{k} p^k \right )^r.
\]
Let $\delta >0$ and let $n_0 = n_0(\delta)$ be as in Corollary \ref{cor:splitting_times_conc}. Suppose that $p = n^{-\lambda^\star/(k \lambda)} / \omega(n)$ for some $\omega(n) \to \infty$ (note that this clearly holds if we have $p \ll n^{-\lambda^\star/(k \lambda)} (\log n)^{-1/rk}$). By Lemma \ref{lem:degree_upper_bound} and the union bound, there is some $C = C(\delta,n_0+|\C{0}{}{0}|)$ such that the  probability that all of the vertices in 
$\C{0}{}{0} \cup \{ 1, 2, \ldots, n_0\}$ have degree at most $C n^{\lambda^\star / \lambda}$ is at least $1- \delta$. Hence, the probability that at least one of them becomes infected in the first step of the bootstrap process is at most
\begin{align}
\label{eqn:early_vertices}
 & \Prob[\C{0}{}{}] { (A_1 \setminus A_0) \cap  \left(\C{0}{}{0} \cup \{1, \ldots, n_0\} \right) \neq \emptyset } 
 \nonumber \\
 & \qquad \leq \delta + (n_0+|\C{0}{}{0}|) \left ( C n^{\lambda^\star / \lambda} \binom{d}{k} p^{k} \right )^r \nonumber \\
 & \qquad \leq \delta + (n_0+|\C{0}{}{0}|) C' \left ( n^{\lambda^\star / \lambda} n^{-\lambda^\star/ \lambda} (\omega(n))^{-k} \right )^r \nonumber \\
 & \qquad = \delta + o(1).
\end{align}
For the vertices arriving after time $n_0$, we use \eqref{eq:rthmom_degree} to bound the probability that at least one of them becomes infected in the first step of the process. We have
\begin{align*}
 & \Prob[\C{0}{}{}] { (A_1 \setminus A_0) \cap \{n_0+1, \ldots, n \} \neq \emptyset } \\
 & \qquad \leq  \Prob[\C{0}{}{}] {\mathcal{E}_{\eps,n_0,n}^c}  +
 \Prob[\C{0}{}{}] {\mathcal{E}_{\eps,n_0,n}, \{ (A_1 \setminus A_0) \cap \{n_0+1, \ldots, n \} \neq \emptyset \} } \\
 &  \qquad \leq \Prob[\C{0}{}{}] {\mathcal{E}_{\eps,n_0,n}^c} + \sum_{i=n_0+1}^{n} \Prob[\C{0}{}{}] { \mathcal{E}_{\eps,n_0,n} , i \in A_1 \setminus A_0 } \\
 &  \qquad = \Prob[\C{0}{}{}] {\mathcal{E}_{\eps,n_0,n}^c} + \sum_{i=n_0+1}^{n} \E[\C{0}{}{}] { \mathbf 1 (\mathcal{E}_{\eps, n_0,n}) \mathbf 1 ( i \in A_1 \setminus A_0) }.
\end{align*}
Let $I^i_j$ be the event that the $j$th $r$-tuple of disjoint (excluding $i$) $k$-dimensional faces in $\St{i}{\C{n}{}{}}$ is critical with respect to the initially infected set. By Corollary \ref{cor:splitting_times_conc} we have $1 - \Prob[\C{0}{}{}] {\mathcal{E}_{\eps,n_0,n}} \leq \delta$, so we continue with
\begin{align*}
 \Prob[\C{0}{}{}] { (A_1 \setminus A_0) \cap \{n_0+1, \ldots, n \} \neq \emptyset } 
 & \leq \delta + \sum_{i=n_0+1}^{n} \E[\C{0}{}{}] { \mathbf 1 (\mathcal{E}_{\eps, n_0,n}) \mathbf 1 ( i \in A_1 \setminus A_0) ) } \\
 & \leq \delta + \sum_{i=n_0+1}^{n} \E[\C{0}{}{}] { \mathbf 1 (\mathcal{E}_{\eps, n_0,n}) \sum_{j = 1}^{\binom{d}{k}^r D_n^r(i)} \mathbf 1 (I^i_j) } \\
 & = \delta + \sum_{i=n_0+1}^{n} \E[\C{0}{}{}] { \mathbf 1 (\mathcal{E}_{\eps, n_0,n}) \binom{d}{k}^r D_n^r(i) \mathbf 1 (I^i_1) } \\
 & = \delta + \sum_{i=n_0+1}^{n} \E[\C{0}{}{}] { \mathbf 1 (\mathcal{E}_{\eps, n_0,n}) D_n^r(i) } \binom{d}{k}^r p^{rk}  \\
 & \stackrel{\eqref{eq:rthmom_degree}}{\leq} \delta + \sum_{i=n_0+1}^{n} C' \cdot \left(\frac{n}{i}\right)^{\frac{r\lambda^\star}{\lambda}} p^{rk},
\end{align*}
for some constant $C' > 0$. Assume first that $p = n^{-\lambda^\star/(k \lambda)} / \omega(n)$ and
$\frac{r\lambda^\star}{\lambda} > 1$. We have
\begin{align}
\label{eqn:late_vertices}
 \Prob[\C{0}{}{}] { (A_1 \setminus A_0) \cap \{n_0+1, \ldots, n \} \neq \emptyset } & \leq \delta + \sum_{i=n_0+1}^{n} C' \cdot \left(\frac{n}{i}\right)^{\frac{r\lambda^\star}{\lambda}} n^{\frac{-r k \lambda^\star}{k \lambda}} (\omega(n))^{-rk} \nonumber \\
 & \leq \delta + C'  (\omega(n))^{-rk} \sum_{i=1}^{\infty} i^{-\frac{r\lambda^\star}{\lambda}} \nonumber \\
 & \leq \delta + C''  (\omega(n))^{-rk}  = \delta + o(1).
\end{align}
Suppose now that $\frac{r\lambda^\star}{\lambda} = 1$ and $p = n^{-\lambda^\star/(k \lambda)} (\log n)^{-1/rk} / \omega(n)$. Following the above argument, we obtain
\begin{align}
\label{eqn:late_vertices_equality}
 \Prob[\C{0}{}{}] { (A_1 \setminus A_0) \cap \{n_0+1, \ldots, n \} \neq \emptyset } & \leq \delta + \sum_{i=n_0+1}^{n} C' \cdot \frac{n}{i} (n \log n)^{-1} (\omega(n))^{-rk} \nonumber \\
 & \leq \delta + C'  (\omega(n))^{-rk} \log^{-1} n \sum_{i=1}^{n} i^{-1} \nonumber \\
 & \leq \delta + C''  (\omega(n))^{-rk}  = \delta + o(1).
\end{align}
By \eqref{eqn:early_vertices} and \eqref{eqn:late_vertices} (or \eqref{eqn:late_vertices_equality}, if $\frac{r\lambda^\star}{\lambda} = 1$), since $\delta > 0$ was arbitrary, we see that no vertex becomes infected in the first step of the process and consequently w.h.p. we do not have percolation.
\end{proof}

\section{Examples}
\label{sec:examples}

The value of the ratio $\frac{r \lambda^\star}{\lambda}$, crucial to Theorem \ref{thm:main}, depends not only on the distribution $\mu$ of the vertex weights, but also on the fitness function $f$. For that reason, designing a universal tool verifying whether $\frac{r \lambda^\star}{\lambda} > 1$ for given families of pairs $(\mu,f)$ appears to be a difficult task. However, in the case of $d$-dimensional Random Apollonian Networks with all $(d-1)$-dimensional faces having equal fitness, say $\gamma > 0$, we have $\lambda = (d-1) \gamma$ and $\lambda = (d-2) \gamma$, so consequently $\frac{r \lambda^\star}{\lambda} = r \frac{d-2}{d-1}$. For $d \geq 4$ this is larger than 1 for all $r \geq 2$, while for $d=3$ the ratio equals $1$ for $r=2$ and is larger than $1$ for all $r \geq 3$, so we can apply Theorem \ref{thm:main} to these models.

Let us consider the following example of model \textbf{B}, where the selected face is removed from further consideration, constituting a generalisation of the Random Apollonian Networks with $d=3$. Assume that we have a two-point distribution of the weights of the vertices of the complex, i.e., that for some $0< \alpha \leq 1$ and $0 < \beta < 1$ we have $\mu(1) = \beta = 1-\mu(\alpha)$, and that the fitness function is $f(x_0,x_1,x_2) = x_0+x_1+x_2$ (note that taking $\alpha = 1$ recovers the Random Apollonian Network model in three dimensions).

This leads to four types of the $2$-dimensional faces: $(\alpha,\alpha,\alpha), (\alpha,\alpha,1), (\alpha,1,1)$, and $(1,1,1)$. Recall from Section \ref{sec:polya_urns} that $\lambda$ is the largest (real) eigenvalue of the matrix $A$ defined as $A_{i,j} = a_i B_{ij}$, where $a_i$ is the fitness of the face of type $i$, and $B_{i,j}$ is the expected number of faces of type $j$ produced by a face of type $i$ upon subdivision (minus 1 if $i=j$ to account for the face that is being subdivided). Hence we have $a_1 = 3\alpha, a_2 = 2\alpha+1, a_3 = \alpha+2$, and $\alpha_4 = 3$, as well as
\begin{align*}
B & = 
\begin{bmatrix}
    3(1-\beta) & 3\beta & 0 & 0 \\
    1-\beta & \beta+2(1-\beta) & 2\beta & 0 \\
    0 & 2(1-\beta) & 2\beta+(1-\beta) & \beta \\
    0 & 0 & 3(1-\beta) & 3\beta
\end{bmatrix}
-I \\ 
& =
\begin{bmatrix}
    2-3\beta & 3\beta & 0 & 0 \\
    1-\beta & 1-\beta & 2\beta & 0 \\
    0 & 2(1-\beta) & \beta & \beta \\
    0 & 0 & 3(1-\beta) & 3\beta-1
\end{bmatrix}.
\end{align*}
Hence, we have
\[
A =
\begin{bmatrix}
    3\alpha(2-3\beta) & 9\alpha\beta & 0 & 0 \\
    (2\alpha+1)(1-\beta) & (2\alpha+1)(1-\beta) & 2(2\alpha+1)\beta & 0 \\
    0 & 2(\alpha+2)(1-\beta) & (\alpha+2)\beta & (\alpha+2)\beta \\
    0 & 0 & 9(1-\beta) & 3(3\beta-1)
\end{bmatrix},
\]
and the largest eigenvalue of $A$ is
\[
 \lambda = \frac{8\alpha+7\beta(1-\alpha)+1+\sqrt{16\alpha^2-8\alpha+1+ \alpha^2 \beta^2 - 2\alpha\beta^2+\beta^2+14\beta+2\alpha\beta-16\alpha^2 \beta}}{2}.
\]
This is a real eigenvalue since for all $0 < \alpha \leq 1$
\begin{align*}
16\alpha^2- 8\alpha+1+ \alpha^2 \beta^2 & - 2\alpha\beta^2+\beta^2+14\beta+2\alpha\beta-16\alpha^2 \beta \\
 & = (4\alpha-1)^2 + \beta^2(\alpha-1)^2+\beta(14+2\alpha-16\alpha^2) \geq 0.
\end{align*}
We make the following claim. 
\begin{clm} \label{clm:lambdas_comp} We have $\lambda_1 \geq \lambda_\alpha$. 
\end{clm}
\begin{proof}[Sketch of the proof of Claim~\ref{clm:lambdas_comp}] 
Consider a $2$-dimensional simplex $\sigma_0$.
Let $S_0^{(x)}$ and $S_0^{(x')}$ be the simplicial complexes that are the result of the 
subdivision of $\sigma_0$ by a vertex of weight $\alpha$ and $1$, respectively. 
Next, consider two $\star$-processes $(S_t^{(\alpha)})_{t \geq 0}$ and $(S_t^{(1)})_{t\geq 0}$.
embedded into continuous time. 
This means that to each face $\sigma $ a Poisson process of rate $f(\sigma)$ 
is associated up until the moment $\sigma$ is subdivided. 
Moreover, these processes are independent. 
We can represent the evolution of these $\star$-processes by ternary trees. 
Each time a face is subdivided, this event is represented by the birth of two children of the leaf that
corresponds to this face. Only leaves give birth and to each leaf a Poisson point process with rate equal 
to the fitness of the corresponding face is associated. 

Let $\mathcal{T}_t^{\alpha}$ and $\mathcal{T}_t^1$ be the corresponding ternary trees. 
Using a common sequence of subdividing vertices, one can couple them so that a.s. in the coupling space $\mathcal{T}_t^{\alpha} \subset \mathcal{T}_t^1$, for 
all $t\geq 0$. 

 Now, observe that $f$ is such that when a face is subdivided, the total fitness of the two faces it produces is 
 greater than the fitness of the subdivided face. As the total fitness is the total fitness of the faces which correspond 
to leaves, it follows that a.s. 
$Z^{\star} (S_t^{(\alpha)}) \leq Z^{\star}(S_t^{(1)})$. The theory of generalised P\'olya urns shows that 
$Z^{\star} (S_t^{(\alpha)}) e^{-\lambda_{\alpha} t}$ and $Z^{\star} (S_t^{(\alpha)}) e^{-\lambda_{1} t}$ have 
finite a.s. limits. Hence, it cannot be the case that $\lambda_\alpha > \lambda_1$. 

\end{proof}

Hence, to find the value of $\lambda^\star$ we analyse the $1\star$-process. We can have only three types of faces containing a vertex of weight 1, namely $(\alpha,\alpha,1), (\alpha,1,1)$, and $(1,1,1)$, with fitnesses $a_1^* = 2\alpha+1, a_2^* = \alpha+2$, and $a_3^* = 3$ respectively. Remembering that in the companion process we only keep faces that contain the fixed centre of the star, we have
\begin{align*}
B^* & = 
\begin{bmatrix}
    2(1-\beta) & 2\beta & 0 \\
    1-\beta & \beta+(1-\beta) & \beta \\
    0 & 2(1-\beta) & 2\beta
\end{bmatrix}
-I \\ 
& =
\begin{bmatrix}
    1-2\beta & 2\beta & 0 \\
    1-\beta & 0 & \beta \\
    0 & 2(1-\beta) & 2\beta-1
\end{bmatrix},
\end{align*}
and consequently
\[
A^* =
\begin{bmatrix}
    (1+2\alpha)(1-2\beta) & 2(1+2\alpha)\beta & 0 \\
    (2+\alpha)(1-\beta) & 0 & (2+\alpha)\beta \\
    0 & 6(1-\beta) & 3(2\beta-1)
\end{bmatrix}.
\]
The largest eigenvalue of $A^*$ is
\[
 \lambda^\star = 1+2\alpha(1-\beta)+2\beta.
\]
We want to verify whether
\begin{align*}
 \frac{\lambda^\star}{\lambda} & = 
 \frac{2+4\alpha(1-\beta)+4\beta}{8\alpha+7\beta(1-\alpha)+1+\sqrt{16\alpha^2-8\alpha+1+ \alpha^2 \beta^2 - 2\alpha\beta^2+\beta^2+14\beta+2\alpha\beta-16\alpha^2 \beta}} \\
  & \geq \frac{1}{2},
\end{align*}
which is equivalent to testing the condition
\begin{equation}
\label{eqn:ratio_condition}
 \sqrt{16\alpha^2-8\alpha+1+ \alpha^2 \beta^2 - 2\alpha\beta^2+\beta^2+14\beta+2\alpha\beta-16\alpha^2 \beta} \leq 3+\beta-\alpha\beta.
\end{equation}
By taking squares we see that equality in~\eqref{eqn:ratio_condition} can only hold when
\begin{align*}
 0 & = 16\alpha^2 - 16\alpha^2\beta+8\alpha\beta-8\alpha+8\beta-8 = 8(1-\beta)(2\alpha^2-\alpha-1) \\
  & = (1-\beta)(2\beta+1)(1-\alpha),
\end{align*}
i.e., when $\beta=1$, $\alpha=1$, or $\beta = -1/2$. Only the $\alpha=1$ solution lies in our domain, and it corresponds to the Random Apollonian Network model discussed already at the beginning of this section. Since our formulae for $\lambda, \lambda^\star$ are continuous in $\alpha, \beta$, it now suffices to observe that plugging in $\alpha = \beta=0$ into \eqref{eqn:ratio_condition} gives $1 < 3$, and so the strict inequality must hold for all $0 < \alpha, \beta < 1$.

\section{Open problems and further generalisations}
\label{sec:further}

The condition $\tfrac{r \lambda^\star}{\lambda} > 1$ in Theorem~\ref{thm:main} allows us to use a first moment argument to show that if $p \ll p_c$ then w.h.p. the initially infected set is stable, i.e., after the initial infection is seeded, no further infections occur. If this condition does not hold then we can expect the initially infected set to grow, but when is the growth substantial enough to cause percolation?

Another interesting (however quite general) question could ask about the properties of the weight distribution $\mu$ and the fitness function $f$ that ensure good lower bounds on the ratio $\lambda^\star / \lambda$, so that Theorem~\ref{thm:main} can be applied (at least for $r$ large enough). In Section~\ref{sec:examples} we show that for $d=3$, all two-point distributions on $\{\alpha,1\}$, with $\alpha \in (0,1]$, and $f(x,y,z) = x+y+z$, lead to the bound $\lambda^\star / \lambda \geq 1/2$ (with equality if and only if $\alpha = 1$, i.e., the weight distribution has mass on one point only). A more general theory related to this problem would be welcome.

Furthermore, in this paper we assume that $\mu$ has finite support. Thus, for example, in the proof of the 
supercritical case we rely on the appearance of vertices of maximum weight. However, this would not work in the
case where $\mu$ is continuous and has not an atom at the maximum of its support. It seems that in this case one would need a different argument.  

Finally, in our process we only infect a new vertex $v$ if its link contains at least $r$ fully infected \emph{disjoint} $(k-1)$-dimensional faces. What if we drop the requirement that the faces need to be disjoint? To be able to infect the last vertex to arrive, we still need to ask for the condition $\binom{d}{k} \geq r$ to hold (which, for any $k < d$, is a strictly weaker condition than $d \geq r k$). However, allowing the infected faces in the link to intersect introduces a lot of difficulties, the most significant of which appears to be the need to control the rate of growth of the codegrees (recall that the ratio $\lambda^\star / \lambda$ gives the rate of growth of the largest degrees). 
The recent paper of Morrison and Noel~\cite{ar:MorNo2018} treats this model of bootstrap process on hypergraphs which certain regularity conditions on the co-degrees. 
In our context, at the moment this appears to be a challenging task.

\bibliographystyle{plain}

\end{document}